\documentclass[12pt]{article}
\usepackage{amsfonts}
\usepackage{amsmath,amsthm,amssymb}
\usepackage{mathtools}
\usepackage[all,dvips]{xy}
\usepackage{eucal}
\usepackage{comment}
\usepackage{tikz-cd}
\usepackage{CJKutf8}

\renewcommand{\deg}{\operatorname{deg}}

\renewcommand{\min}{\operatorname{min}}
\renewcommand{\max}{\operatorname{max}}
\newtheorem{theorem}{Theorem}[section]
\newtheorem{lemma}[theorem]{Lemma}
\newtheorem{corollary}[theorem]{Corollary}

\newtheorem{proposition}[theorem]{Proposition}
\newtheorem{defn}[theorem]{Definition}
\theoremstyle{definition}

\newtheorem{remark}[theorem]{Remark}

\newtheorem{example}[theorem]{Example}

\newcommand{\PE}{\mathbb{P}(E)}
\newcommand{\Vol}{\mathrm{Vol}}
\newcommand{\blup}{\mathrm{Bl}_{\mathbb{P}(E/E_{1})}\mathbb{P}(E)}
\newcommand{\PEE}{\mathbb{P}(E/E_{1})}
\newcommand{\rHE}{\rho^{*}H_{E}}
\newcommand{\rF}{\rho^{*}F}
\newcommand{\N}{N_{\mathbb{P}(E/E_{1})/\PE}}
\newcommand{\mumin}{\mu_{\min}}
\newcommand{\mumax}{\mu_{\max}}
\numberwithin{equation}{section}

\AtEndDocument{\bigskip{\footnotesize%
  \textsc{Département de Mathématiques, UQAM, C.P. 8888, Succursale Centre-ville, Montréal 
(Québec), H3C 3P8, Canada}   
}

\textit{Email address}: \texttt{benammar\_ammar.houari@courrier.uqam.ca}

\textit{Email address}: \texttt{massonnet.louis@courrier.uqam.ca}

\textit{Email address}: \texttt{yin.chenxi@courrier.uqam.ca}}

\author{Houari Benammar Ammar, Louis Massonnet, Chenxi Yin}

\date{}

\title{Delta-invariant for projective bundles over a curve and K-semistability}

\begin{document}

\begin{CJK}{UTF8}{min}

\maketitle

\begin{abstract}
Consider $E$ a vector bundle over a smooth curve $C$. We compute the $\delta$-invariant of all ample ($\mathbb{Q}$-)line bundles on $\mathbb{P}(E)$ when $E$ is strictly Mumford semistable. We also investigate the case when one assumes that the Harder–Narasimhan filtration of $E$ has only one step.
\end{abstract}

\section{Introduction}

This paper is about $\delta-$invariant on a projective bundle over a smooth curve.

The $\delta-$invariant is defined for a big ($\mathbb{Q-}$)line bundle $L$ on a klt projective variety. When $L$ is ample, this invariant contains crucial information related to K-stability, a stability notion introduced in \cite{TianKEpositive1997} and refined in \cite{Donaldsonstability2002}. K-stability is conjectured to be equivalent to the existence of canonical metrics like constant scalar curvature Kähler metrics (CSC Kähler metrics) inside the class $c_{1}(L)$. Especially, we have the following famous result from \cite{BlumXupolydegeneration2019,Fujitavaluative2019, ChiLivolumemin2017}.

\begin{theorem}[\cite{BlumXupolydegeneration2019,Fujitavaluative2019,ChiLivolumemin2017}]
Let \( X \) be a klt Fano variety. The following hold:
\begin{enumerate}
    \item \( (X,-K_{X}) \) is uniform K-stable if and only if \( \delta(X,-K_{X})>1 \). 
    \item \( (X,-K_{X}) \) is K-semistable if and only if \( \delta(X,-K_{X})\geq 1 \).
\end{enumerate}
\end{theorem}

The invariant $\delta(X,-K_{X})$ has been investigated in detail when $X$ is Fano. We can only mention a few studies here. For Fano threefolds, an exhaustive reference is \cite{calabithree}. For projective bundle of Fano type, we refer to the paper of Zhang and Zhou \cite{ZhangZhouDeltabundle2022}. In the important paper \cite{abbanzhuang}, Abban and Zhuang develop an inductive method of computing $\delta-$invariant based on inversion of adjunction, and use this method to show that a large family of hypersurfaces in projective spaces are $K-$stable.

In recent years, the research on $\delta(X,L)$ for ample $L$ has been active. In \cite{Zhangcontinuitydelta2021}, Zhang introduces an analytic version for the $\delta-$invariant, namely $\delta^{A}$, which is defined as the coercivity threshold of the entropy functional in Kähler geometry. Then Zhang shows in the subsequent paper \cite{Zhangquantization2024} that, 

\begin{theorem}[\cite{Zhangquantization2024}]\label{zhangquantization2024}
    The equality $\delta(X,L) = \delta^{A}(X,L)$ holds for any ample line bundle \( L \).
\end{theorem}

This theorem can be used to produce sufficient conditions for an ample line bundle to admit CSC Kähler metrics or being uniform K-stable.

It is also proved by Zhang in \cite{Zhangcontinuitydelta2021} that

\begin{theorem}[\cite{Zhangcontinuitydelta2021}]
The $\delta$-invariant \( \delta(X,\cdot) \) is a continuous function on the big cone.
\end{theorem}

When $X$ is a fibration over a smooth curve, Hattori shows the following theorem in \cite{hattori2022k}.

\begin{theorem}[\cite{hattori2022k}]
Let \( f : (X,  H) \to (C, L) \) be a polarized algebraic fiber space pair. Suppose that \( X \) is klt, \( C \) is a smooth curve, and \( \deg L = 1 \). Then
\[
\lim_{\epsilon \to 0} \delta(X,\epsilon H + f^*L) = 2 \inf_{p \in C} \mathrm{lct}(X,f^{-1}(p)).
\]
\end{theorem}

These results lead us to wonder whether it is possible to find the continuous function \( \delta(X, \cdot) \) on the ample cone, and if so, what the continuous function looks like. Moreover, when \( X \) is a fibration over a smooth curve, what is the asymptotic behavior of the \( \delta \)-invariant when the polarization primarily comes from the base curve? Finally, is it possible to look for K-semistable or uniform K-stable classes using information from the $\delta-$invariant?

With these questions in mind, we start to consider $\delta-$invariant for ample line bundles on a projective bundle over a smooth curve. 
The main results of this article are the following.

\begin{theorem}\label{maintheorem1}
    Assume that $E$ is strictly semistable of $\mathrm{rank}$ $n$ and  slope $\mu$. For an ample line bundle on $\PE$ numerically equivalent to $aH_{E}+bF$, where $H_{E} = \mathcal{O}_{\PE}(1)$ and $F$ is a fiber, we have

    \[ \delta(X,L) = \min\left(\frac{2}{a\mu+b},\frac{n}{a}\right).\]

    If $\mu = 0$, the formula becomes
    \[\delta(X,L) = \min\left(\frac{2}{b},\frac{n}{a}\right).\]
\end{theorem}

\begin{remark}
    Notice that this last formula is compatible with the formula in \cite{zhuang2020product} for a product of varieties.
\end{remark}

We also have the following theorem which concerns a more complicated situation.

\begin{theorem}\label{maintheorem2}
    Let $C$ be a smooth curve and $E$ a vector bundle of rank $n$ and slope $\mu$ over $C$. We assume also that the Harder-Narasimhan filtration of $E$ has only one step,
    \[0=E_{0}\subset E_{1}\subset E_{2}=E\]
    and $E_{1}$ is of rank $r$. Let $\mu_{\max}$ be the slope of $E_{1}$, $\mu_{\min}$ be the slope of $E/E_{1}$. Let $H_{E} = \mathcal{O}_{\PE}(1)$ and $\xi$ be the corresponding numerical class. Let $F$ be a fiber and $f$ be the corresponding numerical class. Let $L$ be an ample line bundle on $X=\PE$ of which the numerical class is $a\xi+bf$. Then the expected vanishing order of $L$ with respect to $F$ is,

\[S(L,F) = \frac{\frac{n(a\mu_{min}+b)(a(2\mu-\mu_{\min})+b)}{2} + a^{2}(\mu_{\max} - \mu_{\min})^{2} \frac{r(r+1)}{2(n+1)}}{n(a\mu+b)}.\]
Setting
\begin{align*}
    &s_{1} = \frac{n(n+1)(a\mu+b)}{a((r+1)(a\mu_{\max}+b)+(n-r)(a\mu_{\min}+b))}\\
    &s_{2} = \frac{n(a\mu+b)}{a(a\mu_{\max}+b)}
\end{align*}
We have,
\[\min\left(\frac{1}{S(L,F)},s_{1}\right) \geq \delta(X,L) \geq\min\left(\frac{1}{S(L,F)},s_{2}\right).\]
If we fix $a$ and $b>>0$ (actually the effective lower bound for $b$ can be provided), 
\[\min\left(\frac{1}{S(L,F)},s_{2}\right)
=\frac{1}{S(L,F)},\]
and in this case,
\[\delta(X,L) = \frac{1}{S(L,F)}.\]
\end{theorem}

With these pieces of information, we can find some K-semistable classes on $\PE$, where $E$ is a rank 2 semistable vector bundle of degree $0$ over an smooth elliptic curve (see Example \ref{exksemi}).

\subsection*{Organization of the paper}
\begin{itemize}
    \item In section 2, we provide the basic setup. We recall the definitions of the Harder-Narasimhan filtration for vector bundles over a curve, nef and pseudo-effective cones of projective bundles over a curve, and the $\delta-$invariant of big line bundles.
    \item In section 3, we use Abban-Zhuang method to get a lower bound for $\delta-$invariant of ample line bundles over a projective bundle.
    \item In section 4, we use the Zariski decomposition developed by Nakayama (\cite{nakayama2004zariski}), along with intersection theory and combinatorics, to provide upper bounds for the $\delta-$invariant when there is only one non-trivial subbundle in the Harder-Narasimhan filtration of the vector bundle. Consequently, we get Theorem \ref{maintheorem2}.
    \item In section 5, we give the formula of $\delta-$invariant of ample line bundles on a projective bundle in Theorem \ref{maintheorem1}, when the corresponding vector bundle is strictly semistable. This formula has applications in the theory of K-stability. We also provide a classical deformation argument to demonstrate that every ample line bundle on a projective bundle over a curve is K-semistable if and only if the associated vector bundle is Mumford semistable. 
\end{itemize}

\subsection*{Acknowledgments}

The authors would like to express their gratitude to their supervisors Vestislav Apostolov, Nathan Grieve, Julien Keller, and Steven Lu for valuable discussions. We would also like to thank Ruadhai Dervan for helpful communications and Xia Xiao for discussions.

\section{Basic setup}

Let $C$ be a smooth curve of genus $g$. Let $E$ be a vector bundle of rank $n$ and slope $\mu$ over $C$. We denote the projective bundle of hyperplanes of $E$ as $\mathbb{P}(E)$, and the projective bundle of lines of $E$ as $\mathcal{P}(E)$. So we have $\mathbb{P}(E) = \mathcal{P}(E^{*})$. We have the Harder-Narasimhan filtration \cite{harder1975cohomology} (see also \cite{nakayama2004zariski}
)

\[0 = E_{0}\subset E_{1}\subset \cdots \subset E_{l} = E.\]

We recall that the Harder-Narasimhan filtration is the unique filtration of sub-bundles of $E$ with the following properties

\begin{itemize}
    \item $\forall i; 1\leq i \leq l$, the bundles $E_{i}/E_{i-1}$ are semistable.\\
    \item $\forall i; 1\leq i \leq l-1$, we have  $\mu(E_{i}/E_{i-1})>\mu(E_{i+1}/E_{i}).$
\end{itemize}

We denote $\mu_{\max} = \mu(E_{1})$, $\mu_{\min} = \mu(E_{l}/E_{l-1})$ and $\mu = \mu(E)$, where $\mu(E)$ represents the slope of $E$. 

Let $X = \PE$, $H_{E} = \mathcal{O}_{\PE}(1)$ and $F$ be a fiber. Let $\xi$ and $f$ be the corresponding numerical classes of $H_{E}$ and $F$. Then it is known that $N^{1}(X) = \mathbb{R}\xi\oplus \mathbb{R}f$. The following result is well-known.

\begin{theorem}[\cite{effeccone},\cite{Miyaokaminimalvariety1987},\cite{nakayama2004zariski}]
The nef cone of $X=\PE$ is generated by $f$ and $\xi-\mu_{\min}f$, the pseudo-effective cone of $X$ is generated by $f$ ad $\xi-\mu_{\max}f$. 
    
\end{theorem}

Now let $L$ be an ample line bundle on $X$, then the numerical class of $L$, denoted by $[L]$, is $[L] = a\xi + bf$. Here $a>0$ and $\frac{-b}{a}<\mu_{\min}$.

In this paper, we want to get information about the $\delta-$invariant of $L$.
 
\begin{defn}[\cite{fujita2018kstability}]
     Let $L$ be a big line bundle on a klt projective variety $X$, then the $\delta$-invariant of $L$ is

     \[\delta(X,L) := \inf_{D/X}\frac{A_{X}(D)}{S(L,D)},\]
     where $D$ is a prime divisor over $X$, meaning that there is a proper birational morphism $f:Y\rightarrow X$ between normal varieties and $D$ is a prime divisor on $Y$. The function $A_{X}$ is the log discrepancy function. The expected vanishing order of $L$ with respect to $D$ is defined as
     \[S(L,D) := \frac{1}{\mathrm{Vol(L)}}\int_{0}^{\tau}\mathrm{Vol}(f^{*}(L)-tD)dt,\]
     where $\tau$ is the pseudo-effective threshold of $f^{*}(L)$ with respect to $D$.

     If $Z$ is a closed subvariety of $X$, then we define

     \[\delta_{Z}(X,L):= \inf_{\substack{D/X,\\ Z\subset C_{X}(D)}}\frac{A_{X}(D)}{S(L,D)},\]
     where $C_{X}(D)$ is simply the image of $D$ on $X$.
\end{defn}

\begin{remark}
The expected vanishing order of $L$ with respect to $D$ has been widely used in arithmetic geometry. For example, in \cite{Nathanexpectation2023}, Grieve explains it from a more arithmetic perspective.
\end{remark}

\section{A lower bound by Abban-Zhuang}

We use the method of Abban-Zhuang in this part. We first recall some of the definitions and results of \cite{abbanzhuang}.

\begin{defn}[\cite{abbanzhuang}]
   Let $X$ be a klt projective variety of dimension $n$. Let $Y$ be a Cartier prime divisor of $X$ which is itself a klt projective variety. Let $L$ be a big line bundle on $X$ and $L_{Y} = L|_{Y}$, $M = -Y|_{Y}$. Denote $V_{m} := H^{0}(X,mL)$ and $V_{\cdot}^{X} := \bigoplus_{m\geq 0}V_{m}$. The refinement of $V_{\cdot}^{X}$ by $Y$ is defined as,
   \begin{align*}
   &W_{\cdot,\cdot}^{Y} := \bigoplus_{m,j\geq 0}W_{m,j}^{Y},\\
   &W_{m,j}^{Y} := \mathrm{Im}(H^{0}(X,mL-jY)\rightarrow H^{0}(Y,mL_{Y}+jM)).
   \end{align*}
   For a prime divisor $D$ over Y, we define a filtration of $W_{\cdot,\cdot}^{Y}$,
   \[\mathcal{F}^{t}_{D}W_{\cdot,\cdot}^{Y} := \bigoplus_{m,j\geq 0}\mathcal{F}_{D}^{mt}W_{m,j}^{Y},\]
   where $\mathcal{F}_{D}^{mt}W_{m,j}^{Y} = \{s\in W_{m,j}^{Y}|\mathrm{ord}_{D}(s)\geq mt\}.$ So for each $t\geq 0$, we have a $\mathbb{Z}^{2}_{\geq 0}-$graded linear series.

   The volume function is defined as follows
   \[\mathrm{Vol}(\mathcal{F}^{t}_{D}W_{\cdot,\cdot}^{Y}) := \lim_{m\rightarrow \infty}\frac{\sum_{j\geq 0}\mathrm{dim}(\mathcal{F}^{mt}_{D}W_{m,j}^{Y})}{m^{n}/n!}.\]

   Especially, $\mathrm{Vol}(W_{\cdot,\cdot}^{Y}) = \mathrm{Vol}(\mathcal{F}^{0}_{D}W_{\cdot,\cdot}^{Y})$ which is actually just $\mathrm{Vol}(L)$ by simple consideration. 

   We define then the expected vanishing order function
   \[S(W_{\cdot,\cdot}^{Y};D) := \frac{1}{\mathrm{Vol}(W_{\cdot,\cdot}^{Y})}\int_{0}^{\infty}\mathrm{Vol}(\mathcal{F}^{t}_{D}W_{\cdot,\cdot}^{Y})dt\]

   Furthermore, we define for closed $Z\subset Y$

   \[\delta_{Z}(Y,W_{\cdot,\cdot}^{Y}) := \inf_{\substack{D/Y,\\ Z\subset C_{Y}(D)}}\frac{A_{Y}(D)}{S(W^{Y}_{\cdot,\cdot};D)}.\]
\end{defn}

One result of \cite{abbanzhuang} (see also \cite[Section 1.7]{calabithree}) is the following.

\begin{theorem} [\cite{abbanzhuang}]\label{abbanzhuang}
    Let $X$ and $Y$ as above and $L$ be an ample line bundle on $X$. Let $Z$ be an irreducible closed subvariety of $Y$. Then we have

    \[\delta_{Z}(X,L)\geq \min\left(\frac{1}{S_{L}(Y)},\delta_{Z}(Y,W_{\cdot,\cdot}^{Y})\right) \text{.} \]
\end{theorem}

Now let us go back to our case $X = \PE$ and $[L] = a\xi + bf$. Notice that $\mathrm{Vol}(L) = na^{n-1}(a\mu+b)$.

Of course we have

\[\delta(X,L) = \inf_{P\in X}\delta_{P}(X,L).\]

Now fix a (closed) point $P\in X$. We have a unique fiber $F\cong \mathbb{P}^{n-1}$ which goes through $P$. Now by Theorem \ref{abbanzhuang}, we have

\[\delta_{P}(X,L) \geq \min\left(\frac{1}{S(L,F)},\delta_{P}(F,W^{F}_{\cdot,\cdot})\right).\]

We try to understand $\delta_{P}(F,W^{F}_{\cdot,\cdot}))$.

\begin{proposition}\label{deltap}
    For the setup as above, we have
    \[\frac{n(a\mu+b)}{a(a\mu_{\max}+b)}\leq \delta_{P}(F,W^{F}_{\cdot,\cdot}))\leq \frac{n(a\mu+b)}{a(a\mu_{\min}+b)}.\]
\end{proposition}

\begin{proof}
    We have $mL-jF|_{F} = \mathcal{O}_{\mathbb{P}^{n-1}}(am)$. So 

    \[W_{m,j}^{F} = \mathrm{Im}(H^{0}(X,mL-jF))\rightarrow H^{0}(\mathbb{P}^{n-1},\mathcal{O}(am))\subset  H^{0}(\mathbb{P}^{n-1},\mathcal{O}(am)).\]

We only need to consider $0\leq j\leq a\mu_{\max}m+bm$ since we understand the pseudoeffective cone of $X$. For a prime divisor $D$ over $F$, we have

\[\sum_{j}\mathrm{dim}\mathcal{F}_{D}^{mt}W_{m,j}^{F} \leq ([a\mu_{\max}m+bm]+1)\mathrm{dim}\{s\in H^{0}(\mathbb{P}^{n-1},\mathcal{O}(am))|\mathrm{ord}_{D}(s)\geq mt\}.\]

So
\begin{align*}
    \Vol(\mathcal{F}_{D}^{t}W_{\cdot,\cdot}^{F}) &= \lim_{m\rightarrow \infty}\frac{\sum_{j\geq 0}\mathrm{dim}(\mathcal{F}_{D}^{mt}W_{m,j}^{F})}{m^{n}/n!}\\
    &\leq \lim_{m\rightarrow \infty}\frac{1}{m^{n}/n!}([a\mu_{\max}m+bm]+1)\mathrm{dim}\{s\in H^{0}(\mathbb{P}^{n-1},\mathcal{O}(am))|\mathrm{ord}_{D}(s)\geq mt\}\\
   &= n\cdot\lim_{m\rightarrow \infty} \frac{[a\mu_{\max}m+bm]+1}{m}\cdot \frac{\mathrm{dim}\{s\in H^{0}(\mathbb{P}^{n-1},\mathcal{O}(am))|\mathrm{ord}_{D}(s)\geq mt\}}{m^{n-1}/(n-1)!}\\
   &= n\cdot (a\mu_{\max} + b)\Vol(\mathcal{O}_{\mathbb{P}^{n-1}}(a)-tD) \text{.}
\end{align*}

Then

\begin{align*}
    S(W_{\cdot,\cdot}^{F},D) &= \frac{1}{\Vol(L)}\int_{0}^{\infty}\Vol(\mathcal{F}_{D}^{t}W_{\cdot,\cdot}^{F})dt\\
    &\leq \frac{1}{\Vol(L)}\int_{0}^{\infty}n\cdot(a\mu_{\max}+b)\Vol(\mathcal{O}_{\mathbb{P}^{n-1}}(a)-tD)\\
    &= \frac{n(a\mu_{\max}+b)}{\Vol(L)}\cdot \Vol(\mathcal{O}_{\mathbb{P}^{n-1}}(a))\cdot S(\mathcal{O}_{\mathbb{P}^{n-1}}(a),D).
\end{align*}

So 

\begin{align*}
  \delta_{P}(F,W^{F}_{\cdot,\cdot})) &= \inf_{\substack{D/F,\\ P\in C_{F}(D)}}\frac{A_{F}(D)}{S(W_{\cdot,\cdot}^{F},D)}\\
  &\geq \frac{1}{n(a\mu_{\max}+b)}\cdot\inf_{\substack{D/F,\\ P\in C_{F}(D)}}\frac{\Vol(L)}{\Vol(\mathcal{O}_{\mathbb{P}^{n-1}}(a))}\frac{A_{F}(D)}{S(\mathcal{O}_{\mathbb{P}^{n-1}}(a),D)}\\
  & = \frac{1}{n(a\mu_{\max}+b)}\cdot\frac{\Vol(L)}{\Vol(\mathcal{O}_{\mathbb{P}^{n-1}}(a))}\delta_{P}(\mathbb{P}^{n-1},\mathcal{O}(a)) \text{.}\\
\end{align*}

It is known that $\delta(\mathbb{P}^{n-1},\mathcal{O}(1)) = n$ by \cite{thresholds2020}. By symmetry we know that $\delta_{P}(\mathbb{P}^{n-1},\mathcal{O}(a)) = \delta(\mathbb{P}^{n-1},\mathcal{O}(a))$. So $\delta_{P}(\mathbb{P}^{n-1},\mathcal{O}(a)) = \frac{n}{a}$. Thus we have

\[\frac{n(a\mu+b)}{a(a\mu_{\max}+b)}\leq \delta_{P}(F,W^{F}_{\cdot,\cdot})).\]

To have the other inequality, we look at the short exact sequence

\[0\rightarrow \mathcal{O}_{X}(mL-(j+1)F)\rightarrow \mathcal{O}_{X}(mL-jF)\rightarrow \mathcal{O}_{\mathbb{P}^{n-1}}(am)\rightarrow 0\]

Then we consider the corresponding long exact sequence. We see easily that, if $H^{1}(X,mL-(j+1)F) = 0$, we have 

\[W_{m,j}^{F} = \mathrm{Im}(H^{0}(X,mL-jF))\rightarrow H^{0}(\mathbb{P}^{n-1},\mathcal{O}(am)) = H^{0}(\mathbb{P}^{n-1},\mathcal{O}(am)) \text{.}\]

We know that

\[K_{X} = -nH_{E} + \pi^{*}(K_{C}+det(E)),\]
where $\pi:\PE\rightarrow C$ is the natural projection. So, $[K_{X}] = -n\xi + (2g-2+d)f$ and $[mL-(j+1)F] = (am+n)\xi + (bm+1-j-2g-n\mu)f + [K_{X}].$ Then by Kodaira vanishing theorem, when $j< \mu_{\min}(am+n)+bm+1-2g-n\mu$, we have $H^{1}(X,mL-(j+1)F) = 0$. This implies that 

\[\sum_{j}\mathrm{dim}\mathcal{F}_{D}^{mt}W_{m,j}^{F} \geq [(a\mu_{\min}+b)m+n\mu_{\min}+1-2g-n\mu]\mathrm{dim}\{s\in H^{0}(\mathbb{P}^{n-1},\mathcal{O}(am))|\mathrm{ord}_{D}(s)\geq mt\} \]
for $D$ over $F$.

Finally we can follow the method above to see that
\[\delta_{P}(F,W^{F}_{\cdot,\cdot}))\leq \frac{n(a\mu+b)}{a(a\mu_{\min}+b)}.\]
\end{proof}

Since for every point $P$, we have

\[\delta_{P}(X,L)\geq\min\left(\frac{1}{S(L,F)},\delta_{P}(F,W^{F}_{\cdot,\cdot})\right)\geq \min\left(\frac{1}{S(L,F)},\frac{n(a\mu+b)}{a(a\mu_{\max}+b)}\right),\]
and $S(L,F)$ is independent of the fiber $F$, so we get

\begin{proposition}\label{AZmethod}
    With the above setup, we have

    \[\delta(X,L)\geq \min\left(\frac{1}{S(L,F)},\frac{n(a\mu+b)}{a(a\mu_{\max}+b)}\right).\]
\end{proposition}

A simple corollary of Proposition \ref{deltap} is

\begin{corollary}\label{sstable}
    If $E$ is semistable, with setup above, we have

    \[\delta_{P}(F,W^{F}_{\cdot,\cdot})) = \frac{\Vol(L)}{a^{n}(a\mu+b)} = \frac{n}{a} \text{.}\]
\end{corollary}

\section{Zariski decomposition for projective bundles and $S(L,F)$}

From the definition of the $\delta-$invariant, we have:

\[\delta(X,L)\leq \frac{A_{X}(F)}{S(L,F)} = \frac{1}{S(L,F)}.\]

Notice that $S(L,F)$ also appears in the lower bound from the last section, so it is clear that it is very important to understand $S(L,F)$. Follow the spirit of \cite{calabithree}, we want to use Zariski decomposition to compute volumes and then to compute $S(L,F)$. Hence we need to understand Zariski decomposition for a projective bundle over a curve. This is discussed in detail in \cite{nakayama2004zariski}. We recall some of the results in \cite{nakayama2004zariski}.

Since the general case can be really complicated and hard to compute, we only discuss an easy case here, namely $l=2$. That is to say, the Harder-Narasimhan filtration of $E$ is

\[0=E_{0}\subset E_{1} \subset E_{2}=E.\]

We have the following diagram

\[
\begin{tikzcd}
 & \mathrm{Bl}_{\mathbb{P}(E/E_{1})}\PE \arrow[dl,"\rho"] \arrow[dr,"\pi"] \\
\PE  \arrow[rr,dotted] && \mathbb{P}(E_{1})
\end{tikzcd}
\]

Let $D$ be the exception divisor of $\rho$ (in the case where $E_{1}$ is of rank 1, $D$ is just $\PEE$), $H_{E} = \mathcal{O}_{\PE}(1)$, $H_{E_{1}} = \mathcal{O}_{\mathbb{P}(E_{1})}(1)$. We have the relation $\rho^{*}(H_{E}) = \pi^{*}(H_{E_{1}}) + D$. The following proposition is from \cite{nakayama2004zariski}.

\begin{proposition}[\cite{nakayama2004zariski}]\label{zaridec}
    Every pseudo-effective $\mathbb{R}-$divisor on $X = \PE$ admits a Zariski decomposition. For $H_{E}-tF$ with $t\leq \mu_{\min}$, the divisor itself is nef. For $H_{E}-tF$ with $\mu_{\min}\leq t\leq \mu_{\max}$, the positive part of its Zariski decomposition is $\alpha\pi^{*}(H_{E_{1}}-\mu_{\max}F) + (1-\alpha)\rho^{*}(H_{E}-\mu_{\min}F)$, where $\alpha$ satisfies $t = \alpha\mu_{\max}+(1-\alpha)\mu_{\min}.$ The negative part is $\alpha D$.
\end{proposition}

Now we have to understand the intersection theory on $\mathrm{Bl}_{\mathbb{P}(E/E_{1})}\mathbb{P}(E)$ to compute volumes. We have to understand intersection numbers of the form $\rho^{*}(H_{E})^{a}\cdot \rho^{*}(F)^{b}\cdot D^{c}$, where $a+b+c=n.$ We always have $\rho^{*}(F)^{2}=0$, so we only need to consider $b=0,a+c=n$ or $b=1,a+c=n-1$. 

Now let $i_{D}:D\hookrightarrow \mathrm{Bl}_{\mathbb{P}(E/E_{1})}\mathbb{P}(E)$ be the closed embedding of the exceptional divisor and $\varphi:D\rightarrow \mathbb{P}(E/E_{1})$ be the projection. We know that $D\cong \mathcal{P}(N_{\mathbb{P}(E/E_{1})/\PE})$. We remark again that we use $\mathbb{P}$ to represent projective bundles of hyperplanes, whereas $\mathcal{P}$ is used to represent projective bundles of lines. We know that $\mathcal{O}_{\blup}(D)\cong \mathcal{O}_{\blup}(-1)$. Because $\mathrm{Proj}$ construction behaves well under base change, we have $i_{D}^{*}(\mathcal{O}_{\blup}(D))\cong \mathcal{O}_{\mathcal{P}(N_{\mathbb{P}(E/E_{1})/\PE})}(-1)$.

The following lemma should be standard and known to experts.

\begin{lemma}
    Let $0\rightarrow E_{1}\rightarrow E\rightarrow E/E_{1}\rightarrow 0$ be an exact sequence of vector bundles over a smooth projective variety $C$. Let $p_{1}:\PE\rightarrow C$ and $p_{2}:\mathbb{P}(E/E_{1})\rightarrow C$ be the corresponding projections. Then the normal bundle $N_{\mathbb{P}(E/E_{1})/\PE}$ is isomorphic to $p_{2}^{*}(E_{1}^{*})\otimes \mathcal{O}_{\mathbb{P}(E/E_{1})}(1)$.
\end{lemma}

\begin{proof}
    We have the relative Euler sequence for $\mathbb{P}(E)$

    \[0\rightarrow \mathcal{O}_{\mathbb{P}(E)}\rightarrow p_{1}^{*}E^{*}\otimes \mathcal{O}_{\PE}(1)\rightarrow T_{\PE/C}\rightarrow 0.\]

    Let $j:\PEE\hookrightarrow \PE$ be the closed embedding. We restrict the above exact sequence to $\PEE$ and get

    \[0\rightarrow \mathcal{O}_{\PEE}\rightarrow p_{2}^{*}E^{*}\otimes \mathcal{O}_{\PEE}(1)\rightarrow j^{*}(T_{\PE/C})\rightarrow 0 \text{.}\]

    We also have the relative Euler sequence for $\mathbb{P}(E/E_{1})$

    \[0\rightarrow \mathcal{O}_{\mathbb{P}(E/E_{1})}\rightarrow p_{2}^{*}(E/E_{1})^{*}\otimes \mathcal{O}_{\PEE}(1)\rightarrow T_{\PEE/C}\rightarrow 0.\]

    Then we have the following natural diagram:

    \begin{tikzcd}
    &0\arrow{r} &\mathcal{O}_{\mathbb{P}(E/E_{1})}\arrow{r} \arrow{d} &p_{2}^{*}(E/E_{1})^{*}\otimes \mathcal{O}_{\PEE}(1)\arrow{r} \arrow{d} &T_{\PEE/C}\arrow{r} \arrow{d} &0\\
    &0\arrow{r} &\mathcal{O}_{\PEE}\arrow{r} &p_{2}^{*}E^{*}\otimes \mathcal{O}_{\PEE}(1)\arrow{r} &j^{*}(T_{\PE/C})\arrow{r} &0
    \end{tikzcd}

    Finally we get what we claimed by snake lemma.
\end{proof}

Now we start to compute $S(L,F)$. We will also compute $S(L,D)$ so that we can have another natural upper bound for $\delta(L).$

What is essential in computing $S(L,F)$ is to compute the following integral

\[\int_{\mu_{\min}}^{\mu_{\max}}\mathrm{Vol}(H_{E}-tF)dt = \int_{0}^{\mu_{\max}-\mu_{\min}}\mathrm{Vol}(H_{E}-(t+\mu_{\min})F)dt.\]

The Zariski decomposition tells us that the positive part of $H_{E}-(t+\mu_{\min})F$ is 
\[\rho^{*}H_{E}-(t+\mu_{\min})\rho^{*}F-\frac{t}{\mu_{\max}-\mu_{\min}}D.\] 

So we need to compute the self-intersection

\begin{align*}
    &(\rho^{*}H_{E}-(t+\mu_{\min})\rho^{*}F-\frac{t}{\mu_{\max}-\mu_{\min}}D)^{n}\\
    =& \sum_{a+b+c=n} \binom{n}{a, b, c}(-(t+\mu_{\min}))^{b}\left(-\frac{t}{\mu_{\max}-\mu_{\min}}\right)^{c}(\rho^{*}H_{E})^{a}(\rho^{*}F)^{b}D^{c}.
\end{align*}

\subsection{Case $c=0$}

In this case, we have $(\rHE)^{n} = n\mu$ and $(\rHE)^{n-1}\rF = 1.$ All the rest are 0.

Correspondingly, we have
\[\int_{0}^{\mu_{\max}-\mu_{\min}}(\rHE)^{n}dt = \int_{0}^{\mu_{\max}-\mu_{\min}}n\mu\text{ }dt = n\mu(\mu_{max}-\mu_{\min})\]
and
\begin{align*}
    \int_{0}^{\mu_{\max}-\mu_{\min}}n(-t-\mu_{\min})(\rHE)^{n-1}\rF dt &= \int_{0}^{\mu_{\max}-\mu_{\min}}n(-t-\mu_{\min})dt\\ &= -n(\mu_{\max}-\mu_{\min})(\frac{\mu_{\max}+\mu_{\min}}{2}) \text{.}
\end{align*}

Putting these two terms together we get

\[\int_{0}^{\mu_{\max}-\mu_{\min}}(\rHE)^{n}dt+\int_{0}^{\mu_{\max}-\mu_{\min}}n(-t-\mu_{\min})(\rHE)^{n-1}\rF dt = (r-\frac{n}{2})(\mu_{\max}-\mu_{\min})^{2}.\]

\subsection{Case $c\neq 0$}

We have the following diagram
\[
\begin{tikzcd}
    &\mathcal{P}(N_{\mathbb{P}(E/E_{1})/\PE})\cong D \arrow[r,"i_{D}"]\arrow[d,"\rho"] &\blup\arrow[d,"\rho"] &\\
    & \PEE \arrow[r,"i"] &\PE  \arrow[r,"p_{1}"] &C
\end{tikzcd}
\]

Then we have

\begin{align*}
    (\rho^{*}H_{E})^{a}(\rho^{*}F)^{b}D^{c} &= (i_{D}^{*}\rHE)^{a}(i_{D}^{*}\rF)^{b}(i_{D}^{*}(\mathcal{O}_{\blup}(D)))^{c-1}\\
    &= (i_{D}^{*}\rHE)^{a}(i_{D}^{*}\rF)^{b}(\mathcal{O}_{\mathcal{P}(N_{\mathbb{P}(E/E_{1})/\PE})}(-1))^{c-1}\\
    &= (-1)^{c-1}(i_{D}^{*}\rHE)^{a}(i_{D}^{*}\rF)^{b}(\mathcal{O}_{\mathcal{P}(N_{\mathbb{P}(E/E_{1})/\PE})}(1))^{c-1}\\
    &= (-1)^{c-1}(i^{*}H_{E})^{a}(i^{*}F)^{b}S_{c-r}(\N),
\end{align*}
where $S_{c-r}$ represents the $(c-r)-$th Segre class. Notice that the above formula is $0$ if $c<r$. Let's denote the morphism between $\PEE$ and $C$ as $p_{2}$ again. Then

\begin{align*}
    S_{c-r}(\N) &= S_{c-r}(p_{2}^{*}(E_{1}^{*})\otimes \PEE(1))\\
    &= \sum_{l=0}^{c-r}(-1)^{c-r-l}\binom{c-1}{r-1+l}S_{l}(p_{2}^{*}E_{1}^{*})c_{1}(\mathcal{O}_{\PEE}(1))^{c-r-l} \text{.}\\
\end{align*}
It follows that,
\begin{align*}
   &(\rho^{*}H_{E})^{a}(\rho^{*}F)^{b}D^{c}\\
   =&(-1)^{c-1}(i^{*}H_{E})^{a}(i^{*}F)^{b}S_{c-r}(\N)\\= &(-1)^{c-1}(i^{*}H_{E})^{a}(i^{*}F)^{b}\sum_{l=0}^{c-r}(-1)^{c-r-l}\binom{c-1}{r-1+l}S_{l}(p_{2}^{*}E_{1}^{*})c_{1}(\mathcal{O}_{\PEE}(1))^{c-r-l}\\
   =& \sum_{l=0}^{c-r}(-1)^{1+r+l}\binom{c-1}{r-1+l}(i^{*}H_{E})^{a+c-r-l}(i^{*}F)^{b}S_{l}(p_{2}^{*}E_{1}^{*})\\
   =&\sum_{l=0}^{c-r}(-1)^{1+r+l}\binom{c-1}{r-1+l}p_{2*}((i^{*}H_{E})^{a+c-r-l}(i^{*}F)^{b})S_{l}(E_{1}^{*}) \text{.}
\end{align*}
Now clearly we need to only consider $l=0$ or $l=1$.

When $b=0$, we see that
\begin{align*}
    (\rho^{*}H_{E})^{n-c}D^{c} = &(-1)^{1+r}\binom{c-1}{r-1}p_{2*}((i^{*}H_{E})^{n-r}) + (-1)^{r}\binom{c-1}{r}p_{2*}((i^{*}H_{E})^{n-r-1})S_{1}(E^{*})\\
    =& (-1)^{1+r}\binom{c-1}{r-1}\mathrm{deg}(E/E_{1}) + (-1)^{r}\binom{c-1}{r}\mathrm{deg}(E_{1})\\
    =&(-1)^{1+r}\binom{c-1}{r-1}(n-r)\mumin + (-1)^{r}\binom{c-1}{r}r\mumax \text{.}
\end{align*}

The corresponding integral is

\begin{align*}
    &\int_{0}^{\mumax-\mumin}\binom{n}{n-c,0,c}(\frac{-t}{\mumax-\mumin})^{c}((-1)^{1+r}\binom{c-1}{r-1}(n-r)\mumin + (-1)^{r}\binom{c-1}{r}r\mumax)dt\\
    =&\binom{n}{c}\frac{(-1)^{c+r}(\mumax-\mumin)}{c+1}(\binom{c-1}{r}r\mumax-\binom{c-1}{r-1}(n-r)\mumin) \text{.}
\end{align*}

When $b=1$, remark that $p_{2*}((i^{*}H_{E})^{n-r-2}i^{*}F)=0$. So, we have

\begin{align*}
     (\rho^{*}H_{E})^{n-c-1}(\rho^{*}F)D^{c} = &(-1)^{1+r}\binom{c-1}{r-1}p_{2*}((i^{*}H_{E})^{n-r-1}i^{*}F) + (-1)^{r}\binom{c-1}{r}p_{2*}((i^{*}H_{E})^{n-r-2}i^{*}F)S_{1}(E^{*})\\
     =&(-1)^{r+1}\binom{c-1}{r-1} \text{.}
\end{align*}

The corresponding integral is

\begin{align*}
    &\int_{0}^{\mumax-\mumin}\binom{n}{n-c-1,1,c}(-t-\mumin)(\frac{-t}{\mumax-\mumin})^{c}(-1)^{r+1}\binom{c-1}{r-1}dt\\
    =& \binom{n}{n-c-1,1,c}(-1)^{c+r}\frac{(\mumax-\mumin)^{2}}{c+2} + \binom{n}{n-c-1,1,c}(-1)^{c+r}\frac{\mumax-\mumin}{c+1}\mumin\binom{c-1}{r-1} \text{.}
\end{align*}

We can take the sum of the above two integrals for $b=0$ and $b=1$ together and then simplify to get

\[    (-1)^{c+r}(\mumax-\mumin)^{2}(\frac{1}{c+2}\binom{n}{n-c-1,1,c}\binom{c-1}{r-1} +\frac{1}{c+1}\binom{n}{c}\binom{c-2}{r-1}(c-1)) \text{.}\]

\subsection{A combinatorial identity}

In total, we have

\begin{align*}
    &\int_{0}^{\mu_{\max}-\mu_{\min}}\mathrm{Vol}(H_{E}-(t+\mu_{\min})F)dt\\ =&(r-\frac{n}{2})(\mu_{\max}-\mu_{\min})^{2}  + \\&\sum_{c=r}^{n} (-1)^{c+r}(\mumax-\mumin)^{2}(\frac{1}{c+2}\binom{n}{n-c-1,1,c}\binom{c-1}{r-1} +\frac{1}{c+1}\binom{n}{c}\binom{c-2}{r-1}(c-1)) \text{.}\\ 
\end{align*}

Notice that we have

\begin{align*}
&(-1)^{c+r}\frac{1}{c+2}\binom{n}{n-c-1,1,c}\binom{c-1}{r-1} + (-1)^{c+r+1}\binom{n}{c+1}\binom{c-1}{r-1}c\\ = &(-1)^{c+r+1}\frac{1}{c+2}\binom{c-1}{r-1}\binom{n}{c+1}.
\end{align*}

So we can simplify to 

\begin{lemma}

We have
    \begin{align*}
&\int_{0}^{\mu_{\max}-\mu_{\min}}\mathrm{Vol}(H_{E}-(t+\mu_{\min})F)dt\\ 
=&(\mu_{\max}-\mu_{\min})^{2}(r-\frac{n}{2}+\sum_{c=r+1}^{n}\frac{1}{c+1}\binom{c-2}{r-1}\binom{n}{c}(-1)^{c+r+1})\\
    =& (\mu_{\max}-\mu_{\min})^{2}(r-\frac{n}{2}+\sum_{c=r+1}^{n}\frac{1}{n+1}\binom{c-2}{r-1}\binom{n+1}{c+1}(-1)^{c+r+1})\text{.}
\end{align*}
\end{lemma}

The following equality is elementary. It allows us to simplify the above expression.

\begin{lemma}\label{com identity}
    For positive integers $n\geq r+1$, we have

    \[\sum_{c=r+1}^{n}\binom{c-2}{r-1}\binom{n+1}{c+1}(-1)^{c+r+1} = \frac{r(r+1)}{2} + (n+1)(\frac{n}{2}-r).\]
\end{lemma}

\begin{proof}
    To prove this, we use induction on $n$. 

    For $n=r+1$, this is a simple verification. Now we assume that we already have the equality for $n$ and consider the case $n+1.$ Notice that $\binom{n+2}{c+1} = \binom{n+1}{c+1} + \binom{n+1}{c}$. The left hand side for $n+1$ is

    \begin{align*}
        &\sum_{c=r+1}^{n+1}\binom{c-2}{r-1}\binom{n+2}{c+1}(-1)^{c+r+1}\\
        =&  \sum_{c=r+1}^{n+1}\binom{c-2}{r-1}(\binom{n+1}{c+1} + \binom{n+1}{c})(-1)^{c+r+1}\\
        =& \sum_{c=r+1}^{n}\binom{c-2}{r-1}\binom{n+1}{c+1}(-1)^{c+r+1} + \sum_{c=r+1}^{n+1}\binom{c-2}{r-1}\binom{n+1}{c}(-1)^{c+r+1}\\
        =& \frac{r(r+1)}{2} + (n+1)(\frac{n}{2}-r) + \sum_{c=r+1}^{n+1}\binom{c-2}{r-1}\binom{n+1}{c}(-1)^{c+r+1} \text{.}
    \end{align*}

    So we need to prove that

    \[\sum_{c=r+1}^{n+1}\binom{c-2}{r-1}\binom{n+1}{c}(-1)^{c+r+1} = n+1-r,\]

    or equivalently for $n\geq r+2, $

    \[\sum_{c=r+1}^{n}\binom{c-2}{r-1}\binom{n}{c}(-1)^{c+r+1} = n-r.\]

    We do the same procedure of induction with n here. First check the case where $n=r+2$, then do the induction. Then we have to show that for $n\geq r+3$
    
    \[\sum_{c=r+1}^{n}\binom{c-2}{r-1}\binom{n}{c-1}(-1)^{c+r+1} = 1.\]
    
    We do again the same procedure of induction with $n$. At last we need

    \[\sum_{c=r}^{n}\binom{c}{r}\binom{n}{c}(-1)^{c} = 0.\]

    This is true because of $\binom{n}{c}\binom{c}{r} = \binom{n}{r}\binom{n-r}{c-r}$ and the famous equality

    \[\sum_{r=0}^{n}\binom{n}{r}(-1)^{r} = 0.\]
\end{proof}

Then we see that 

\begin{align*}
    &(\mu_{\max}-\mu_{\min})^{2}(r-\frac{n}{2}+\sum_{c=r+1}^{n}\frac{1}{n+1}\binom{c-2}{r-1}\binom{n+1}{c+1}(-1)^{c+r+1})\\
    =& \frac{r(r+1)}{2(n+1)}(\mumax-\mumin)^{2} \text{.}
\end{align*}

\subsection{About $S(L,F)$ and $\delta(X,L)$}

From the consideration above, we get  the following proposition

\begin{proposition}
    We have 

    \[S(L,F) = \frac{\frac{n(a\mu_{min}+b)(a(2\mu-\mu_{\min})+b)}{2}+a^{2}(\mu_{\max}-\mu_{\min})^{2}\frac{r(r+1)}{2(n+1)}}{n(a\mu+b)}.\]
\end{proposition}

It is also very natural to look at $S(L,D).$ We are going to show the following

\[S(L,D) = \frac{ra((r+1)(a\mu_{\max}+b)+(n-r)(a\mu_{\min}+b))}{n(n+1)(a\mu+b)} \text{.}\]

We have the following lemma

\begin{lemma}
    The nef threshold of $\rho^{*}(L)$ with respect to $D$ is the same as the pseudo-effective threshold of $\rho^{*}(L)$ with respect to $D$, which is a. 
\end{lemma}

\begin{proof}
    Notice that $\rho^{*}(aH_{E}+bF) = \pi^{*}(aH_{E_{1}}+bF) + aD$. By our assumption, $a>0$ and $\frac{-b}{a}<\mu_{\min}<\mu_{\max} = \mu(E_{1})$. So $aH_{E_{1}}+bF$ is ample on $\mathbb{P}(E_{1})$ since $E_{1}$ is semistable. So $ \pi^{*}(aH_{E_{1}}+bF)$ is nef. By \cite[Section 3.a.]{nakayama2004zariski}, we know that $\blup$ is actually a projective bundle over $\mathbb{P}(E_{1})$. Then by projection formula we know that
    \[H^{0}(\blup,\mathcal{O}(\pi^{*}(m(aH_{E_{1}}+bF)))) = H^{0}(\mathbb{P}(E_{1}),(m(aH_{E_{1}}+bF))).\]
    So $\pi^{*}(aH_{E_{1}}+bF)$ is not big. This confirms the claim.
\end{proof}

So we know that

\[\int_{0}^{\infty}\mathrm{Vol}(L-tD)dt = \int_{0}^{a}\mathrm{Vol}(L-tD)dt \text{.}\]

By computation, we need to show

\begin{align*}
&a^{n+1}n\mu + na^{n}b + \sum_{c=r}^{n}(-1)^{c}\binom{n}{c}\frac{a^{n+1}}{c+1}((-1)^{r+1}\binom{c-1}{r-1}(n-r)\mu_{\min}+(-1)^{r}\binom{c-1}{r}r\mu_{\max})\\ 
+&\sum_{c=r}^{n-1}\frac{a^{n}b}{c+1}\binom{n}{n-c-1,1,c}(-1)^{c+r+1}\binom{c-1}{r-1}\\
=& \frac{ra^{n}((r+1)(a\mu_{\max}+b)+(n-r)(a\mu_{\min}+b))}{n(n+1)(a\mu+b)} \text{.}
\end{align*}

To do this, we just need to use the following equality from the proof of lemma \ref{com identity}

\[\sum_{c=r+1}^{n}\binom{c-2}{r-1}\binom{n}{c}(-1)^{c+r+1} = n-r.\]

several times. So we get

\begin{proposition}\label{SLD}
    We have
\[S(L,D) = \frac{ar((r+1)(a\mu_{\max}+b)+(n-r)(a\mu_{\min}+b))}{n(n+1)(a\mu+b)}.\]
    
\end{proposition}

Notice that $A_{X}(D) = r$ where $X=\PE$, so $\frac{A_{X}(D)}{S(L,D)} = \frac{n(n+1)(a\mu+b)}{a((r+1)(a\mu_{\max}+b)+(n-r)(a\mu_{\min}+b))}.$

Finally we have,

\begin{proof}[Proof of Theorem \ref{maintheorem2}]
    Combining the information on $\frac{A_{X}(D)}{S(L,D)}$, $\frac{1}{S(L,F)}$ and Proposition \ref{AZmethod} proves the Theorem \ref{maintheorem2}.
\end{proof}
\begin{remark}
    Notice that

    \[\lim_{a\rightarrow 0} \delta(X,L) = \frac{2}{b}.\]

    This is compatible with the result in \cite{hattori2022k}.
\end{remark}

\begin{remark}
It has been brought to our attention that Engberg conducted similar calculations in \cite{engberg2022some} for the $\beta$-invariant of the exceptional divisor $D$. Since he considers the more general case where the base is not necessarily a curve, he only obtains asymptotic information.
\end{remark}

\begin{remark}
    It is notable that \cite{wolfe2005asymptotic,chen2011computing} provide formulas for computing volume functions on a projective bundle over a curve. However, it appears that their formulas are not very straightforward to use for concrete computations. Additionally, they do not provide formulas for volume functions when the exceptional divisor $D$ is taken into consideration. Therefore, we choose to use Zariski decomposition as the main tool in our calculations.
\end{remark}

\section{Semistable vector bundles and K-semistability}

In this section, we consider the case where $l=1$, which is equivalent to say that $E$ is semistable. In this case, $\mu=\mu_{\min} = \mu_{\max}.$ As before, we assume that we have an ample line bundle $L$ on $\PE$ which is numerically equivalent to $aH_{E} + bF$, where $H_{E} = \mathcal{O}_{\PE}(1)$ and $F$ is a fiber. In this case, it is easy to compute $S(L,F)$ and get

\begin{proposition}\label{Prop 5.1}
    When $E$ is semistable, we have

    \[S(L,F) = \frac{a\mu+b}{2}.\] 
\end{proposition}

We focus on the case where $E$ is strictly semistable, meaning that there is a  sub-vector bundle $E'\subset E$ such that $\mu(E') = \mu.$ We assume that the rank of $E'$ is $r$. Then similar to the situation in the last section, we can consider $S(L,D)$, where $D$ is the exceptional divisor of the morphism $\blup\rightarrow \PE$ or $\PEE$ if $E_{1}$ is of rank 1. Using the same method which leads to Proposition \ref{SLD}, we get

\begin{proposition}\label{Prop 5.2}
    When $E$ is strictly semistable, we have
    \[S(L,D) = \frac{ra}{n}.\]
    Since $A_{X}(D) = r$ where $X = \PE$, we have $\frac{A_{X}(D)}{S(L,D)} = \frac{n}{a}.$
\end{proposition}

We are now ready to prove Theorem \ref{maintheorem1}
\begin{proof}[Proof of Theorem \ref{maintheorem1}]
    Combining Corollary \ref{sstable}, Propositions \ref{Prop 5.1} and \ref{Prop 5.2} proves the theorem.
\end{proof}

\begin{remark}\label{remarkstable}
    Notice that if $E$ is a stable vector bundle  of $\mathrm{rank}$ $n$ and slope $\mu$ and $L$ is an ample line bundle on $\PE$ numerically equivalent to $aH_{E}+bF$. From Proposition \ref{AZmethod} and corollary \ref{sstable}, we know that

    \[ \delta(X,L) \geq \min\left(\frac{2}{a\mu+b},\frac{n}{a}\right).\]
    But we don't know how to conclude an equality here. This might be related to the result of \cite{BlumLiulowersemicontinuous} which states that $\delta-$invariant is lower semi-continuous in family. We have the moduli space $\mathcal{M}(n,d)$ of $\mathrm{rank}$ $n$ and degree $d$ semistable vector bundles over $C$. Then we have the universal vector bundle $\mathcal{E}$ over $\mathcal{M}(n,d)\times C$. On $\mathbb{P}(\mathcal{E})$, it is easy to construct a family of ample line bundles $L_{t}$ by using $\mathcal{O}_{\mathbb{P}(\mathcal{E})}(1)$ and pullback of $\mathcal{M}(n,d)\times \{pt\}$. We know that stable bundles form an open subset of $\mathcal{M}(n,d)$, so it might be possible that the $\delta(L_{t})$ jump in this open subset.
\end{remark}

From the work of \cite{Zhangquantization2024}, we know that $\delta(L)$ contains rich analytical information, namely, it is the coercivity threshold of the entropy functional. Using this result and combining techniques from papers like \cite{Dervanstability,Dervancoercivity,LiShiYaocoercivity}, it is possible to find some sufficient criteria  for an ample line bundle to be uniform K-stable or K-semistable. For example, we can have the following proposition which is a slight generalization of the result in \cite{LiShiYaocoercivity}.

\begin{proposition}\label{suffikss}
    Let $L$ be an ample line bundle on an $n-$dimensional manifold X. If we have

    \begin{itemize}
        \item $\delta(X,L)L + K_{X}$ is nef\text{,}\\
        \item $(n\frac{K_{X}\cdot L^{n-1}}{L^{n}}+\delta(X,L))L - (n-1)K_{X}$ is nef.
    \end{itemize}

    Then $L$ is K-semistable.
\end{proposition}

\begin{proof}
We provide a brief proof here for the convenience of the readers. 

Notice that for any $\epsilon>0$, $(\delta(X,L)+\epsilon)L + K_{X}$ and $(n\frac{K_{X}\cdot L^{n-1}}{L^{n}}+\delta(X,L)+\epsilon)L - (n-1)K_{X}$ are ample. The second one can also be written as

\begin{align*}
&n\frac{((\delta(X,L)+\epsilon)L+K_{X})\cdot L^{n-1}}{L^{n}}L - (n-1)((\delta(X,L)+\epsilon)L+K_{X}).
\end{align*}

Then by \cite{WeinkoveJflow2006} and \cite{GaoJequation2021}, we know that the pair $(X,L,(\delta(X,L)+\epsilon)L+K_{X})$ is uniform $J-$stable.

Now let us assume that we have a Kähler test configuration (see for example \cite[Definition 2.10]{DervanRosskstabilitu2017}) $(\mathcal{X},\mathcal{L})$ of $(X,L)$. Then the twisted non-Archimedean $J-$functional (see for example \cite[Definition 6.3]{DervanRosskstabilitu2017}) $J^{\mathrm{NA}}_{(\delta(X,L)+\epsilon)L+K_{X}}(\mathcal{X},\mathcal{L})>0$ for nontrivial $(\mathcal{X},\mathcal{L}).$ Thus,

\[J^{\mathrm{NA}}_{\delta(X,L)L+K_{X}}(\mathcal{X},\mathcal{L}) = \lim_{\epsilon\rightarrow 0}J^{\mathrm{NA}}_{(\delta(X,L)+\epsilon)L+K_{X}}(\mathcal{X},\mathcal{L}) \geq 0 ,\]
since these quantities are computed by intersection theory.

We deduce easily from Theorem \ref{zhangquantization2024} cited in the introduction that 
\[M(\varphi)\geq J_{\delta(X,L)L+K_{X}}(\varphi) + C,\]
where $M$ is the Mabuchi functional, $J_{\delta(X,L)L+K_{X}}$ is the twisted J-functional, $C$ is a constant and $\varphi$ is a Kähler potential. Then we can take non-Archimedean limits (\cite{BHJuniform2017}) of both sides to get

\[M^{\mathrm{NA}}(\mathcal{X},\mathcal{L}) \geq J^{\mathrm{NA}}_{\delta(X,L)L+K_{X}}(\mathcal{X},\mathcal{L})\geq 0\]
\end{proof}

It is easy to get the following example by using  Theorem \ref{maintheorem1}, Remark \ref{remarkstable} and Proposition \ref{sstable}.

\begin{example}\label{exksemi}
    Assume that $C$ is an elliptic curve, $E$ is a $\mathrm{rank}$ 2 strictly semistable vector bundle of degree $0$. We know that Nef cone of $X = \PE$ is the same as the pseudoeffective cone here and is generated by $H_{E} = \mathcal{O}_{\PE}(1)$ and $F$, where $F$ is a fiber. Now for any ample line bundle numerically equivalent to $aH_{E}+bF$ with $a\geq b,$ we know that 
    \[\delta(X,L) = \frac{2}{a}.\]
    Moreover, $K_{X} = -2H_{E}$. Then $\delta(L)L + K_{X}$ is numerically equivalent to $\frac{2b}{a}F$ and $(2\frac{K_{X}\cdot L}{L^{2}}+\delta(L))L - K_{X}$ is numerically equivalent to $2H_{E}.$ So we know that $L$ is K-semistable. 
    We can actually just assume that $E$ is semistable since in this case $\delta(X,L)\geq \frac{2}{b}.$
\end{example}

On the other hand, we can actually show that for a semistable vector bundle $E$ over a smooth curve, every ample line bundle on $\mathbb{P}(E) = \mathcal{P}(E^{*})$ is K-semistable by using a classic deformation statement. 

Let $0 = V_{0}\subset V_{1}\subset \cdots \subset V_{l} = E^{*}$ be the Jordan - Hölder filtration of $E^{*}$ such that $V_{i}/V_{i+1}$ is a stable vector bundle with $\mu(V_{i}/V_{i+1}) = \mu(E).$ We first write down the transition functions for $E^{*}$. 

Assume that we have a covering \( \{ U_\alpha \} \) of \( X \) such that over each \( U_\alpha \), the filtration
\[
0 = V_0 \subset V_1 \subset \cdots \subset V_l = E^{*}
\]
is trivial, meaning we can choose local basis adapted to the filtration.

On overlaps \( U_\alpha \cap U_\beta \), the transition functions \( g_{\alpha\beta} \) of \( E^{*} \) will be block upper-triangular matrices of the form:
\[
g_{\alpha\beta} = \begin{pmatrix}
g_{\alpha\beta}^{(1)} & \xi_{12}^{(\alpha\beta)} & \xi_{13}^{(\alpha\beta)} & \cdots & \xi_{1l}^{(\alpha\beta)} \\
0 & g_{\alpha\beta}^{(2)} & \xi_{23}^{(\alpha\beta)} & \cdots & \xi_{2l}^{(\alpha\beta)} \\
0 & 0 & g_{\alpha\beta}^{(3)} & \cdots & \xi_{3l}^{(\alpha\beta)} \\
\vdots & \vdots & \vdots & \ddots & \vdots \\
0 & 0 & 0 & \cdots & g_{\alpha\beta}^{(l)}
\end{pmatrix},
\]
where \( g_{\alpha\beta}^{(i)} \) are the transition functions of \( V_i / V_{i-1} \), and \( \xi_{ij}^{(\alpha\beta)} \) represents the gluing data of the extension.

To construct the deformation, we introduce a parameter \( t \) (the coordinate on $\mathbb{C}$) and modify the transition functions to be
\[
g_{\alpha\beta}(t) = \begin{pmatrix}
g_{\alpha\beta}^{(1)} & t \xi_{12}^{(\alpha\beta)} & t^2 \xi_{13}^{(\alpha\beta)} & \cdots & t^{l-1} \xi_{1l}^{(\alpha\beta)} \\
0 & g_{\alpha\beta}^{(2)} & t \xi_{23}^{(\alpha\beta)} & \cdots & t^{l-2} \xi_{2l}^{(\alpha\beta)} \\
0 & 0 & g_{\alpha\beta}^{(3)} & \cdots & t^{l-3} \xi_{3l}^{(\alpha\beta)} \\
\vdots & \vdots & \vdots & \ddots & \vdots \\
0 & 0 & 0 & \cdots & g_{\alpha\beta}^{(l)}
\end{pmatrix}.
\]

These transition functions define a holomorphic vector bundle \( \mathcal{E} \) over \( C \times \mathbb{C} \), with the following properties
\begin{itemize}
    \item For \( t \neq 0 \), we get back our $E^{*}$ on $C\times \{t\}$.
    \item At \( t = 0 \), all the off-diagonal terms vanish, and the transition functions reduce to a block-diagonal form
    \[
    g_{\alpha\beta}(0) = \begin{pmatrix}
    g_{\alpha\beta}^{(1)} & 0 & 0 & \cdots & 0 \\
    0 & g_{\alpha\beta}^{(2)} & 0 & \cdots & 0 \\
    0 & 0 & g_{\alpha\beta}^{(3)} & \cdots & 0 \\
    \vdots & \vdots & \vdots & \ddots & \vdots \\
    0 & 0 & 0 & \cdots & g_{\alpha\beta}^{(k)}
    \end{pmatrix},
    \]
    which corresponds to the direct sum bundle \( \operatorname{Gr}(E^{*}) = \bigoplus_{i=1}^{k} V_i / V_{i-1} \) which is polystable.
\end{itemize}

On $\mathcal{P}(\mathcal{E})$, the projective bundle of lines of $\mathcal{E}$, we can consider line bundles $\mathcal{O}_{\mathcal{P}(\mathcal{E})}(a) \otimes \varphi^{*}L$, where $\varphi$ is the natural morphism from $\mathcal{P}(\mathcal{E})\rightarrow C.$ and $L$ is an ample line bundle on $C$ of degree $b$. When $a>0$ and $\frac{-b}{a}<\mu=\mu(E),$ the line bundle is relatively ample over $C.$ This line bundle restricts to the same line bundle on $C$ over any $t\in \mathbb{C}.$ When we restrict the line bundle on the central fiber, we have a CSC Kähler metric in the corresponding Kähler class because of the following theorem from \cite{vestigauduchoncurve2011,VestiJulien2019,ross2006obstruction}.

\begin{theorem}
Let \( V \) be a holomorphic vector bundle on a smooth curve $C$ and let \( X = \mathbb{P}(V) \to C \) be its projectivization. The following three conditions are equivalent
\begin{enumerate}
    \item[(i)] \( X \) admits a CSC Kähler metric in any class \( c_1(L) \).
    \item[(ii)] \( X \) is K-polystable for any polarization \( L \).
    \item[(iii)] \( V \)  is polystable, i.e., decomposes as the sum of stable bundles of the same slopes.
\end{enumerate}
\end{theorem}

This tells us that the Kähler class of $\mathcal{O}_{\PE}(a)\otimes p^{*}L$, where $p:X = \PE\rightarrow C$ is the natural projection, is analytically semistable (\cite{dervan2018moduli}), thus  $\mathcal{O}_{\PE}(a)\otimes p^{*}L$ is K-semistable (also from \cite{dervan2018moduli}). Also notice that if $E$ is not polystable(semistable), none of the ample line bundle on $\PE$ can be K-polystable(K-semistable) as explained in \cite[Theorem 5.13]{ross2006obstruction} and \cite[Remark 1.1]{VestiJulien2019}. 

We conclude as follows.

\begin{theorem}\label{semiksemi}
    For a vector bundle $E$ over a smooth curve $C$. The following are equivalent

    \begin{enumerate}
        \item[(i)] $E$ is semistable but not polystable.
        \item[(ii)] every ample line bundle on $X=\PE$ is K-semistable but not K-polystable.
        \item[(iii)] one ample line bundle on $X=\PE$ is K-semistable but not K-polystable.
    \end{enumerate}
\end{theorem}

\begin{remark}
    Clearly we cannot get Theorem \ref{semiksemi} by using the method in Example \ref{exksemi}. This suggests that the sufficient conditions for the existence of CSC Kähler metric we have at the moment are rather weak. We might need to search for better sufficient conditions by taking into more consideration of the geometry of the variety. One reason for which Proposition \ref{suffikss} is not strong enough is that it depends too much the positivity of $K_{X}$. If $X$ is a fibration over some base, $K_{X}$ can easily have different positivity behaviours at horizontal and vertical directions. It seems more reasonable to try to distinguish these two directions instead of just using $K_{X}$. We leave this to further exploration.
\end{remark}

\begin{remark}
    We want to mention the recent paper of Ortu and Sektnan \cite{ortu2024constant}. In this paper, they also used the deformation argument. They find that, contrary to our cases, when $E$ is a semistable vector bundle with respect to some polarization over a higher dimensional projective manifold, it is possible to find CSC Kähler metrics inside some $c_{1}(L)$ for some ample line bundles on $\PE$. Notice also that the existence of CSC Kähler metrics inside $c_{1}(L)$ implies the K-polystability of $(X,L)$ by \cite{BDLregularity2020}.    
\end{remark}

\bibliographystyle{plain}
\bibliography{References}
\nocite{*}

\end{CJK}

\end{document}